\documentclass[12pt]{article}
\usepackage{stackrel}
\usepackage{amsmath}
\usepackage{amsfonts}
\usepackage{amssymb}
\usepackage{amsthm}
\usepackage{color}

\newtheorem{ex}{Example}

\newtheorem{df}{Definition}[section]
\newtheorem{thm}{Theorem}[section]

\newcommand{\go}[1]{\mathfrak{#1}}
\newcommand{\R}{{\rm I}\kern-0.18em{\rm R}}
\newcommand{\1}{{\rm 1}\kern-0.25em{\rm I}}
\newcommand{\E}{{\rm I}\kern-0.18em{\rm E}}
\newcommand{\p}{{\rm I}\kern-0.18em{\rm P}}


\usepackage{epsfig}

\usepackage{fancyhdr}

\usepackage{graphics}

\usepackage{amsopn}  

\usepackage{amssymb}
\usepackage{amsxtra}
\usepackage{amsfonts}

\usepackage{afterpage}
\usepackage{eucal}
\usepackage{eufrak}

\usepackage{enumerate}

\allowbreak

\def\fnote#1{\footnote}

\newcommand{\bea}{\begin{eqnarray}}

\newcommand{\eea}{\end{eqnarray}}

\newcommand{\beas}{\begin{eqnarray*}}
\newcommand{\eeas}{\end{eqnarray*}}

\title{ $\go G$-casual Stable Probability Distributions}
\author{Lev B. Klebanov\footnote{Department of Probability and Statistics,
Charles University, Prague, Czech Republic}}
\date{}
\begin{document}
\maketitle

\begin{abstract}
A generalization of stable and casual stable probability distribution is proposed. The notion of $\go G$-casual stability can be used to introduce discrete analogues of stable distributions on the sent $\mathbb Z$ of integers. In contrary to limit definition of stable distributions on $\mathbb Z$ our has algebraic character. 
Examples of corresponding limit theorems are given.
\end{abstract}
{\it Keywords}: stable distributions; casual stable distribution; discrete stable distributions; limit theorems

\section{Introduction}
\setcounter{equation}{0}
The role of infinite divisible and stable probability distributions is well known in probability and its applications. It is mainly connected to limit theorems for the sums of independent random variables on one hand and to the stability (or self-similarity) property of such distributions. There are tenth of monographs on this problematics. Therefore, any generalization of the stability and/or infinite divisibility properties is of essential interest. Here a new notion of stability, so-called $\go G$-casual stability, is proposed. It is a generalization of rather new definition of casual stable distributions (see, \cite{KS_CS}), and may be used for analysis of discrete distributions on the set  $\mathbb Z$ of integers. Our further presentation is given on a language of characteristic functions.

\section{Main Definition}
\setcounter{equation}{0}

Let us remind definitions of strictly stable and casual stable characteristic functions. Suppose that $f(t)$ is a characteristic function of non-degenerate random variable. We say $f(t)$ is {\it strictly stable} characteristic function (corresponding probability distribution function is said to be strictly stable distribution) if for any positive integer $n$ there exists $a_n \in (0,1)$ such that 
\begin{equation}\label{eq1}
f(t) = f^n(a_n t), \;\; t\in \R^1. 
\end{equation} 
In this situation there is $\alpha \in (0,2]$ such that  $a_n = 1/n^{1/\alpha}$. To stress attention on the fact that $\alpha$ is fixed we tell about strictly $\alpha$-stable characteristic function (probability distribution function or random variable). The descriptions of strictly $\alpha$-stable random variables can be found in many textbooks on probability theory.

As it was mentioned in Introduction, a new generalization of the stability property (\ref{eq1}) had been recently given in \cite{KS_CS}.  Let us give their definition here.

Let $f(t)$ be a characteristic function of non-degenerate random variable, and $F(x)$ be its probability distribution function. Let us write it in the following form
\[ f(t) = \int_{-\infty}^{\infty}Exp(i t x) dF(x) = \int_{-\infty}^{0}Exp(i t x) dF(x) +\int_{0}^{\infty}Exp(i t x) dF(x) =\]
\[ = f_1(t)+f_2(t). \]
We say $f(t)$ is {\it casual strictly stable} (or $g_n$-casual strictly stable) characteristic function if for any positive integer $n$ there is a characteristic function $g_n(t)$ such that
\begin{equation}\label{eq2}
f(t)= \Bigl(f_1(i\log g_n(-t))+f_2(-i\log g_n(t))\Bigr)^n, \;\; t\in \R^1.
\end{equation}
In (\ref{eq1}) the constants $a_n$ play role of normalization, while in (\ref{eq2}) this role comes to characteristic function $g_n(t)$. The characteristic function $f$ is essentially the same in both sides of (\ref{eq2}). Let us pass to the definition of $\go G$-casual strictly stable characteristic function.

\begin{df}\label{de1}
  Let $f$ be characteristic function of non-degenerate random variable. We say $f(t)$ is $\go G$-casual strictly stable characteristic function if there is a characteristic function ${\go g}(t)$ of non-degenerate random variable and for any positive integer $n$ there exist characteristic function $g_n(t)$ such that
 \begin{equation}\label{eq3}
 f(t) =\Bigl( {\go g_1}(i \log g_n(-t)) + {\go g_2}(-i \log g_n(t)) \Bigr)^n,
 \end{equation} 
 where (in natural notations)
 \[ {\go g}(t) = \int_{-\infty}^{\infty}Exp(i t x) d{\go G}(x) = \int_{-\infty}^{0}Exp(i t x) d{\go G}(x) +\int_{0}^{\infty}Exp(i t x) d{\go G}(x) =\]
\[ = {\go g}_1(t)+{\go g}_2(t). \]
\end{df}
It is clear, $\go G$-casual strictly stable characteristic function is infinitely divisible in classical sense. The main difference with the definition of casual strictly stable characteristic function is that here we use another  characteristic function ${\go g(t)}$  in the right hand side of (\ref{eq3}) instead of $f(t)$ in the left hand side of it. Interpretation of $\go G$-casual strictly stable characteristic function is very similar to that of casual stable case, and we will not discuss it here.

\section{Example of $\go G$-casual Stable Characteristic Functions}
\setcounter{equation}{0}

The notion of discrete stability for lattice random variables on nonnegative integers was introduced in Steutel and van Harn \cite{StH}. Together with a study of discrete self-decomposability they obtained a form of the generating function of such discrete stable distributions. An attempt to give a definition of discrete stable random variable on the set $\mathbb Z$ had been made in \cite{KS_IV}. However, it was possible to give some definitions on the base of limiting properties of stable distributions. Our first Example will provide an algebraic approach to the first of definition of such kind.

\begin{ex}\label{ex1}
Consider a characteristic function 
\begin{equation}\label{eq4}
f(t) = \exp \{-\lambda (1-\cos t )^{\gamma} \},
\end{equation}
where $\gamma \in (0,1]$. In \cite{KS_IV} the function (\ref{eq4}) was called symmetric discrete stable characteristic function. Let us now show it is $\go G$-casual strictly stable characteristic function. To this goal consider characteristic function of positive discrete stable distribution (see \cite{StH})
\begin{equation}\label{eq5} 
{\go g}(t) =\exp\{ -\lambda (1-e^{i t})^{\gamma} \}, 
\end{equation}
and introduce
\[ g_n(t) = (1-1/n^{1/\gamma})+1/n^{1/\gamma} \cos t . \]
It is clear that:
\begin{enumerate}
\item Distribution function $\go G$ corresponding to characteristic function $\go g$ is concentrated on the set of positive integers. According to this, ${\go g}_1(t)=0$ and ${\go g}_2(t)={\go g}(t)$ for all $t \in \R^1$.
\item Functions $g_n(t)$, $n=1,2, \ldots$ are characteristic functions.
\item $f(t)= {\go g}^n(-i \log g_n(t))$ for all $n=1,2, \ldots$.
\end{enumerate}
In other words, $f(t)$ is $\go G$-casual strictly stable characteristic function.
\end{ex}
In the paper \cite{KS_IV} there was given a second limit definition of discrete stability on $\mathbb Z$. It leads to characteristic function of the form
\[ f(t) = \exp\{ -\lambda_1 (1-e^{it})^{\alpha}  -\lambda_2 (1-e^{-it})^{\alpha}\}, \]
 $\lambda_1>0$, $\lambda_2>0$, $\alpha \in (0,1]$. It is possible to show, this characteristic function is $\go G$-casual strictly stable, too. Below we shall show this for the case $\alpha =1/2$.
 \begin{ex}\label{ex2} Let us consider the function
 \[ f(t) = \exp\{ -\lambda_1 \sqrt{1-e^{it})}  -\lambda_2 \sqrt{1-e^{-it})} \} \]
 and show it is $\go G$-casual strictly stable characteristic function. The fact, $f(t)$ is characteristic function follows from considerations in \cite{KS_IV} or may be verified directly.
 
Consider the same characteristic function (\ref{eq5}) as in Example \ref{ex1}, and "normalizing" characteristic functions $g_n(t)$ define by the relation
\[ g_n(t) = 1- \frac{1}{n^{2}}\Bigl( \frac{\lambda_1}{\lambda}\sqrt{1-e^{it}}+\frac{\lambda_2}{\lambda}\sqrt{1-e^{-it}}\Bigr) ^2.\]
The fact \item $f(t)= {\go g}^n(-i \log g_n(t))$ for all $n=1,2, \ldots$ is obvious. We have verify only that $g_n(t)$ is characteristic function for sufficiently large $\lambda$.

If $g_n$ is a characteristic function, then view of its periodical character, it has to be connected to a random variable $\nu$ on $\mathbb Z$. Using inverse formula we find
\[ \p\{\nu=k \} =\frac{8\lambda_1 \lambda_2}{\lambda^2 \pi n^2 (4k^2-1)}  \]
for $k=\pm 1, \pm 2, \ldots$, and
\[ \p\{ \nu =0 \} =1-\frac{8 \lambda_1 \lambda_2}{\lambda^2 n2 \pi}. \]
It is clear  all the probabilities are positive for sufficiently large $\lambda$.
 \end{ex}
 
 Examples \ref{ex1} and \ref{ex2} shows that the limit definitions of the paper \cite{KS_IV} can be considered as $\go G$-casual strictly stable distributions. Now let us turn to some new examples of $\go G$-casual stable characteristic functions.
 
 \begin{ex}\label{ex3} Hermite distribution is a distribution of positive integer random variable. Its probability generating function is
\[ {\mathcal P}(z) =\exp\{ a_1 (z-1)+a_2 (z^2-1) \}, \]
 where $a_1,a_2$ are non-negative constants. Our aim in this Example is to show Hermite distribution is $\go G$-casual strictly stable. Corresponding characteristic function is
\begin{equation}\label{eq6}
 f(t)=\exp \{a_1 (e^{it}-1)+a_2 (e^{2it}-1)\}
\end{equation}
 
 Consider characteristic function of Poisson distribution
 \[ {\go g}(t) = \exp\{ a (e^{it}-1 \}, \]
 and a sequence of functions
 \[ g_n(t) = 1+ \frac{1}{a n} \Bigl( a_1 (e^{it}-1) + a_2 (e^{2it}-1) \Bigr). \]
 If the constant $a>0$ is chosen in such way that 
 \[  0< \frac{a_1}{n a}< 1; \; 0<\frac{a_2}{n a}<1; \; 1-\frac{a_1+a_2}{n a}>0 \]
 for all positive integers $n$, then $g_n(t)$ is characteristic function. Without any difficulties we see that
 \[ f(t) = {\go g}^n(-i \log g_n(t)) \]
 and therefore $f$ is $\go G$-casual strictly stable characteristic function.
\end{ex}
From Example \ref{ex3} it is easy to come to essentially more general case.
\begin{ex}\label{ex4}
Suppose that $h(t)$ is arbitrary characteristic function, and $a>0$ is a constant. It is known that
\[ f(t) =\exp\{a(h(t)-1)\} \]
is infinite divisible characteristic function. Let us show, it is $\go G$-casual strictly stable characteristic function. Really, define characteristic function
\[ {\go g}(t) =\exp\{ A (e^{it}-1)\} \]
for sufficiently large $A>0$. If the function $g_n(t)$, $n=1,2, \ldots$ has the form
\[ g_n(t)= (1-\frac{a}{nA}) +\frac{a}{nA} h(t) \]
then for all $A>a$ $g_n$ is characteristic function and
\[ {\go g}(-i \log g_n(t)) =\exp\{\frac{a}{n} (h(t) -1)\}. \]
Therefore,
\[ f(t) = {\go g}^n(-i \log g_n(t)). \]
The last equation shows $f(t)$ is $\go G$-casual strictly stable characteristic function.
\end{ex}

\section{Some Properties of $\go G$-casual Stable Characteristic Functions}
\setcounter{equation}{0}

We have seen some examples of $\go G$-casual stable characteristic functions. From Example \ref{ex4} it follows the set of such function is large. In this section we give some properties of $\go G$-casual strictly stable distributions. 

\subsection{Non-uniqueness of Representation (\ref{eq3})}
The first question we are interested in is whether the representation (\ref{eq3}) is unique. It appears, the answer is negative. To show this, let us consider characteristic function of positive discrete stable distribution
\[ f(t) = \exp\{ - \lambda (1-e^{it})^{\gamma}\}, \]
$\gamma \in (0,1/2)$. For the first representation define
\[ g_n^{(1)}(t) =1-\frac{1}{n^{1/\gamma}}+\frac{1}{n^{1/\gamma}}e^{it}. \]
Obviously, 
\[ f(t) = {\go g^{(1)}}^n(-i \log g_n^{(1)}(t)), \]
and we have our first variant of representation (\ref{eq3}), where ${\go g^{(1)}}=f$  (more precisely, ${\go g^{(1)}}_1 = 0$ and ${\go g^{(1)}}_2=f$).

For the second representation choose
\[ {\go g^{(2)}}(t) = \exp\{-\lambda  (1-e^{it})^{2\gamma},  \]
and
\[  g_n^{(2)}(t) = 1- \frac{1}{n^{1/(2\gamma)}}(1-e^{it})^{1/2}.  \]
Now it is obvious,
\[ f(t) = {\go g^{(2)}}^n(-i \log g_n^{(2)}(t)), \]
which gives the second representation.

\subsection{Limit Theorem}
The representation (\ref{eq3})
 \[ f(t) =( {\go g_1}(i \log g_n(-t)) + {\go g_2}(-i \log g_n(t)) )^n \]
 shows, that $\go G$-casual strictly stable random variable is a distribution for sums of $n$ independent random variables with characteristic function of the form ${\go g_1}(i \log g_n(-t)) + {\go g_2}(-i \log g_n(t))$, and therefore it is a limit distribution of the sequence of these sums. It is natural to suppose, this distribution is a limit distribution not only for this sequence, but for sequence defined by another characteristic function $\go g$. It is really so under some additional conditions. Below we give a few results of such kind. For simplicity, let us consider the case of functions $\go g$ corresponding to positive random variables. For this situation (\ref{eq3}) takes form
 \begin{equation}\label{eq3a}
  f(t) ={\go g}^n(-i \log g_n(t)) ).
 \end{equation}
 
 What may we expect? Probably, it would be possible to substitute the sequence $\{-i \log g_n(t), \; n \geq 1\}$ under sign of a different function $\go h$ instead of $\go g$, and pass to limit for corresponding powers as $n$ tends to infinity. However, sometimes it is impossible. For example, if $\go g$ is a characteristic function of a positive integer-valued random variable, then ${\go g}(-i \log g_n(t)) )$ is also characteristic function for any characteristic function $g_n$. If $\go h$ corresponds to the random variables taking arbitrary positive values, it is not so, and we have to suppose, that $g_n$ is infinitely divisible characteristic function. It may appear to be very restrictive. Say, in Example \ref{ex4} for the case of entire characteristic function $h$ the function 
 \[ g_n(t)= \Bigl(1-\frac{a}{nA}\Bigr) +\frac{a}{nA} h(t) \] 
 cannot be infinite divisible for all positive integers $n$. To see that it is enough to note, that entire infinite divisible characteristic function cannot have zeros on complex plain. It implies that $h$ cannot take the sequence of values $1-nA/a$, however, such exclusion may be only one for an non-constant entire function. So, for Example \ref{ex4} we may choose $\go h$ among characteristic functions of random variables taking positive integer values.
 
 Our first limit theorem is connected to Example \ref{ex4}. 
 \begin{thm}\label{th1} Suppose that $h(t)$ is a characteristic function. For $\lambda \in (0,1)$ consider a sequence of characteristic functions
\[ g_n(t) = \Bigl(1-\frac{\lambda}{n}\Bigr)+ \frac{\lambda}{n}*h(t). \]
Let $\go h$ is a characteristic function of positive integer random variable with probability generating function ${\mathcal P}(z)$. Suppose further, that there exist the first derivative of the function $\mathcal P$ at point $z=1$. Then
\[ \lim_{n \to \infty}{\go h}^n(-i \log g_n(t)) = \exp\{ a (h(t)-1)\}, \]
where $a=\lambda {\mathcal P}^{\prime}(1)$.
 \end{thm}
 \begin{proof}
 We have
 \[ {\go h}^n(-i \log g_n(t)) = {\mathcal P}^n(g_n(t))= \Bigl( 1- \frac{\lambda}{n}{\mathcal P}^{\prime}(1) (1-h(t)+ o(\lambda/n)\Bigr)^n \] 
 \[ \stackrel[n \to \infty]{}{\longrightarrow} \exp\{-\lambda {\mathcal P}^{\prime}(1) (1-h(t)\}. \]
 Limit characteristic function $f(t) = \exp\{-\lambda {\mathcal P}^{\prime}(1) (1-h(t)\}$ coincides with corresponding function from Example \ref{ex4} with $a=\lambda {\mathcal P}^{\prime}(1)$.
 \end{proof}
 
Because Example \ref{ex3} is a particular case of Example \ref{ex4} we can have corresponding limit theorem for convergence to Hermite distribution. It follows from Theorem \ref{th1} by taking $h(t) = (a_1 e^{it}+a_2 e^{2it})(a_1+a_2)$. We give no precise formulation here. Limit theorems for Examples \ref{ex1} and \ref{ex2} were given in the paper \cite{KS_IV} as definitions of corresponding discrete stability.

\end{document}